\documentclass[a4paper,english, 11pt]{article}

\usepackage[utf8]{inputenc}%caractères accentués
\usepackage[T1]{fontenc}
\usepackage{babel}
\usepackage[babel]{csquotes}

\usepackage{amsmath}%substack, displaystyle
\usepackage{amsfonts}%mathbb, mathfrak, mathcal
\usepackage{amsthm}%newtheorem, proof
\usepackage{bbm}%fonction indicatrice
\usepackage{amssymb}%sphericalangle
\usepackage[top=3.2cm, bottom=3.3cm, left=3.6 cm, right=3.6 cm]{geometry}
\usepackage{mathrsfs}
\usepackage[pdftex, pdfstartview=FitW,colorlinks=true,linkcolor=blue,%
urlcolor=blue,citecolor=blue]{hyperref}
\usepackage{stmaryrd}
\usepackage{cleveref}

\usepackage{float}%choisir l'emplacement des images
\usepackage{tikz}%dessins
\usetikzlibrary{shapes.misc}
\tikzset{cross/.style={cross out, draw, 
minimum size=2*(#1-\pgflinewidth), 
inner sep=0pt, outer sep=0pt}}
\usepackage{IEEEtrantools}%equations trop longues
\usepackage{url}%tilde dans les url

\usepackage{enumitem}%choisir la numération dans enumerate

\newtheorem{th.}{Theorem}[section]
\newtheorem{thm}[th.]{Theorem}
\newtheorem{cnj}[th.]{Conjecture}

\newtheorem{lemma}[th.]{Lemma} 

\newtheorem{prop.}[th.]{Proposition}
\newtheorem{prp}[th.]{Proposition}
\newtheorem{cor.}[th.]{Corollary}
\newtheorem*{rem.}{Remark}

\newcommand{\norm}[1]{\lVert#1\rVert}   % norme

\newcommand{\QQ}{\overline{\mathbb{Q}}}
\newcommand{\Q}{\mathbb{Q}}

\newcommand{\N}{\mathbb{N}}
\newcommand{\Z}{\mathbb{Z}}
\newcommand{\C}{\mathbb{C}}
\newcommand{\R}{\mathbb{R}}

\newcommand{\NN}{\mathscr{N}}
\newcommand{\HH}{\mathcal{H}}
\newcommand{\Su}{\mathscr{S}}
\newcommand{\G}{\mathcal{G}}
\newcommand{\eps}{\varepsilon}
\newcommand{\z}{\underline{z}}

\newcommand{\kt}{\mathfrak{t}}
\newcommand{\kg}{\mathfrak{g}}
\newcommand{\kh}{\mathfrak{h}}

\newcommand {\cO} {{\mathcal O}}

\renewcommand{\a}{\alpha}

\renewcommand{\d}{\delta}
\newcommand{\e}{\varepsilon}

\newcommand{\wt}{\widetilde}

\renewcommand{\a}{\alpha}

\renewcommand{\d}{\delta}

\renewcommand{\L}{\Lambda}

\DeclareMathOperator{\GL}{GL}

\DeclareMathOperator{\dist}{dist}

\DeclareMathOperator{\Aut}{Aut}

\title{Exponential multiple mixing for commuting automorphisms of a nilmanifold}
\author{Timoth\'ee B\'enard and P\'eter P. Varj\'u\thanks{
The authors have received funding from the European Research
Council (ERC) under the European Union's Horizon 2020 research and
innovation programme (grant agreement No. 803711).
PV was supported by the Royal Society.
}}
\date{}

\begin{document}
	
\large

\maketitle

\renewcommand{\thefootnote}{\fnsymbol{footnote}} 
\footnotetext{\textit{Key words and phrases.} Exponential mixing, multiple mixing, nilmanifold automorphisms, Schmidt's subspace theorem, unit equations. }    
\footnotetext{2020 \textit{Mathematics Subject Classification.} Primary 37A25; Secondary 22D40, 11J87.}    
\renewcommand{\thefootnote}{\arabic{footnote}}

{\abstract  Let $l\in \N_{\ge 1}$  and $\alpha : \Z^l\rightarrow  \text{Aut}(\NN)$ be an action of $\Z^l$ by automorphisms on a compact nilmanifold $\NN$. We assume the action of every $\alpha(z)$ is ergodic for $z\in \Z^l\smallsetminus\{0\}$ and show that $\alpha$ satisfies exponential $n$-mixing  for any integer $n\geq 2$. This extends results of Gorodnik
and Spatzier [Acta Math., 215 (2015)].}

\section{Introduction}

Building on the work of Gorodnik and Spatzier \cite{GorSpat14}, \cite{GorSpat15},
the goal of this paper is to show multiple mixing with exponential rate for actions of commuting ergodic automorphisms of a compact nilmanifold. We start by recalling the basic notions involved.

A \emph{compact nilmanifold} is a quotient $\NN=G/\Lambda$ where $G$ is a nilpotent, connected, simply connected real Lie group and $\Lambda$ is a cocompact discrete subgroup of $G$. It carries a unique $G$-invariant probability measure that we denote by $m$ and call the \emph{Haar measure} on $\NN$.
The  \emph{group of automorphisms} of $\NN$ is defined by 
 \[
 \text{Aut}(\NN)=\{\psi \in \text{Aut}(G),\, \psi(\Lambda)=\Lambda\}
 \]
  and acts by measure-preserving transformations on $(\NN, m)$.

This article deals with the mixing properties of a finite family of commuting automorphisms of $\NN$, realized as a morphism $\alpha: \Z^l\rightarrow \Aut(\NN)$. Given $n\geq2$, we say that $\alpha$ is  \emph{$n$-mixing} if  for every bounded measurable functions $f_{1}, \dots, f_{n} : \NN\rightarrow \R$ and for $z_{1}, \dots, z_{n}\in \Z^l$,  we have 
\[
\int_{\NN}\prod_{i=1}^n f_{i}(\alpha(z_{i})x)\,dm(x) \,\,\longrightarrow \,\,\prod_{i=1}^n \int_{\NN} f_{i}(x)\,dm(x)
\]
as $\min_{i\neq j} \norm{z_{i}-z_{j}}\to+\infty$.

Note that the case where $n=2$ corresponds to the usual notion of mixing, and also that we do not lose generality by considering smooth test functions.

Our aim is to estimate the rate of mixing of a finite family of ergodic commuting automorphisms of $\NN$.  To do so, we need to restrict our attention to a set of regular test functions. We fix an arbitrary Riemannian metric on $\NN$, and for $\theta\in (0,1]$,  denote by $\HH^{\theta}(\NN)$ the space of $\theta$-H\"older functions $f$ on $\NN$ endowed with the norm 
\[
\norm{f}_{\HH^{\theta}}=\sup_{\NN}|f|\,+\,\sup_{x\neq y\in \NN}\frac{|f(x)-f(y)|}{d(x,y)^\theta}.
\]

\bigskip
We state our main result.

\begin{thm}[Multiple mixing with exponential rate]\label{TH}
Let $\alpha :\Z^l \rightarrow \emph{Aut}(\NN)$ be an action of $\Z^l$ ($l\geq1$) by automorphisms on a  compact nilmanifold $\NN=G/\Lambda$.  Assume $\alpha(z)$ acts ergodically on $(\NN, m)$ for all nonzero  $z\in \Z^l$. 

Then for every $n\geq 2$, $\theta\in (0,1]$, there is an effective constant $\eta>0$
and an ineffective constant $C>0$ such that for all $\theta$-H\"older functions $f_{1}, \dots, f_{n}\in \HH^{\theta}(\NN)$, translation parameters $g_{1}, \dots, g_{n}\in G$, and $\underline{z}=(z_{1}, \dots, z_{n})\in (\Z^l)^{n}$ one has 
\[
\Big|\int_{\NN}\prod_{i=1}^n f_{i}(g_{i}\alpha(z_{i})x) \,dm(x)\,\, -\,\, \prod_{i=1}^n \int_{\NN}f_{i}(x) \,dm(x)\Big|   \,\,\leq \,\, \frac{C}{N(\underline{z})^{\eta}} \prod_{i=1}^n \norm{f_{i}}_{\HH^\theta} 
\]
where   $N(\underline{z})=\exp (\min_{i\neq j}\norm{z_{i}-z_{j}})$.
\end{thm}

In other words, a finite family of commuting ergodic automorphisms of a compact nilmanifold satisfies multiple mixing with exponential rate in the class of $\theta$-H\"older functions (and their $G$-translates). In particular, we also obtain exponential multiple mixing for actions of commuting \emph{affine} automorphisms. 

We explain what effectivity or ineffectivity of the constants mean in Section
\ref{sc:effective} below.

\subsection{Earlier results}

Theorem \ref{TH} for mixing without a rate, that is the statement that 
ergodicity of $\alpha(z)$ for every $z\in \Z^l\smallsetminus\{0\}$
implies that $\alpha$ is $n$-mixing for all $n\geq 2$,
was known before.
It is due to Parry (1969 \cite{Parry69}) for $l=1$, K. Schmidt-Ward (1993 \cite{SchmWard93})  in the case where   $\NN$ is a torus and $l\geq2$, and Gorodnik-Spatzier (2015 \cite{GorSpat15}) in general. 

There are also many prior results on exponential mixing.
The case of Theorem \ref{TH} where $\NN$ is a torus and $l=1$ is due to  Lind (1982 \cite{Lind82}) for $n=2$, and to  Dolgopyat (2004 \cite{Dolg04}) for general $n$. Note that some difficulty arises from the fact that an ergodic automorphism of a torus is not necessarily hyperbolic: it can have eigenvalues of modulus 1. Still for a torus but now general $l$,  Miles-Ward (2011 \cite{MilWard11}) prove directional uniformity in the --a priori non-exponential-- rate of $2$-mixing under some entropic conditions. For general $\NN$ and  $l\geq1$, the case where $n\le 3$ is established by Gorodnik and Spatzier  in 2015  \cite{GorSpat14, GorSpat15}. They also examine the case of higher order mixing, i.e. $n\geq4$, but only obtain partial results and  mention serious difficulties to reach full generality. Those difficulties consist in a fine understanding of the size of solutions of some diophantine inequalities.

Theorem \ref{TH} is new for $n\ge 4$, even for tori.

\subsection{Effectivity of the constants}\label{sc:effective}

The constant $\eta$ is \emph{effective}.
This means that by following the arguments in this paper and its references, one could
determine an explicit value for $\eta$ such that the theorem holds (for some $C$).
In contrast, the constant $C$ is \emph{ineffective}.
This means that the proof of the theorem only shows that there exists a finite (and positive)
value of $C$ for which the theorem holds; however, we do not have any means of determining such
a value.

This ineffectivity is caused by our reliance on W. Schmidt's subspace theorem.
In the cases $n\le 3$, this tool can be replaced by the theory of linear forms in
logarithms, as it is done by Gorodnik and Spatzier in \cite{GorSpat15} and one can
obtain the theorem with effective constants.
We note that, as they are written, the arguments in \cite{GorSpat15}
appeal to \cite[Theorem 7.3.2]{BomGub07} in several places (see pp. 139--140),
which relies on the subspace theorem making the resulting constants ineffective.
However, the application of \cite[Theorem 7.3.2]{BomGub07} could be replaced by
\cite[Proposition 14.13]{Mas}
at the
expense of adjusting the exponents in \cite{GorSpat15}.

\subsection{Motivation}

Katok and Spatzier have made a conjecture about rigidity of higher rank abelian
Anosov actions on compact manifolds, which can be stated somewhat informally as follows.

\begin{cnj}\label{cn:KS}
	When $l\ge 2$, all irreducible Anosov genuine $\Z^l$-actions on compact manifolds
	are $C^\infty$-conjugate to actions on infranilmanifolds by affine automorphisms.
\end{cnj}

We do not define all the terms appearing in the above conjecture, but we
comment on them briefly.
An Anosov diffeomorphism of a compact manifold is a diffeomorphism such that the tangent
bundle can be split as the sum of two invariant subbundles, with one subbundle that is
exponentially contracting and one that is exponentially expanding under the action.
A $\Z^l$-action is Anosov, if it contains an Anosov diffeomorphism.

The conjecture is false for rank $1$ actions, see \cite[Section 1.2]{RHW} for
a simple example or \cite{FJ} for a more elaborate one.
These can be modified to obtain some higher rank examples, e.g., actions on manifolds
that have quotients on which the action factors through a rank $1$ action.
The adjective genuine is meant to exclude examples like this, but we do not
give a precise definition.

An infranilmanifold is a compact manifold that is finitely covered by a nilmanifold.

Conjecture \ref{cn:KS} motivates Theorem \ref{TH} for two reasons.
First, if the conjecture is true, it implies that actions by affine
automorphisms on (infra)nilmani\-folds are the only examples of higher
rank abelian Anosov actions in some sense, which demonstrates the importance of this
class of dynamical systems.
Second, some recent progress by Fisher, Kalinin and Spatzier \cite{FKS2}
and by Rodriguez Hertz and Wang \cite{RHW}
towards Conjecture \ref{cn:KS} relies on exponential mixing of actions by automorphisms.
We note, however, that these applications require results about $2$-mixing only,
which is already covered in \cite{GorSpat15}.

We refer to \cite{RHW}, \cite{FKS1} and their references for more information on higher rank abelian
Anosov actions.

\subsection{An outline of the paper}

Our approach  is inspired by the papers of Gorodnik and Spatzier \cite{GorSpat14, GorSpat15}. The argument we propose is, however, simpler.

The mixing estimate in Theorem \ref{TH} can be recast as a problem about quantitative
equidistribution of a certain affine subnilmanifold $\Su$ in the product nilmanifold $\NN^n$. In Section \ref{Sec2}, we use a result by Green and Tao \cite[Theorem 1.16]{GreenTao12} to reduce quantitative equidistribution of $\Su$ to a diophantine condition on its Lie algebra. This condition is related to \cite[Theorem 2.3]{GorSpat15} though our formulation is simpler, and easier to prove.

The diophantine condition arising in  Section \ref{Sec2} requires
bounding the solutions of certain generalized unit equations.
This problem is studied in  Section \ref{Sec3}  using W. Schmidt's subspace theorem.
It is related to \cite[Proposition 3.1]{GorSpat15}, but  we prove a more
uniform statement, which is needed to establish the exponential
mixing rate.
Our proof is based on that of a very closely related result of Evertse \cite{Eve}.

In order to apply the results of Section \ref{Sec3} to the diophantine equations governing the equidistribution of $\Su$ in $\NN^n$, we need to estimate the
growth of the eigencharacters of $(\a(z))_{z\in \Z^l}$ acting on the abelianized Lie algebra $\kg/[\kg,\kg]$.
We do this in Section \ref{Sec1} using the ergodicity assumption.

Section \ref{Sec4} concludes the paper with the proof of Theorem \ref{TH}
combining the ingredients worked out in the previous three sections.

\subsection*{Acknowledgement}
We thank the anonymous referee for helpful comments.

\section{Equidistribution of rational submanifolds}
\label{Sec2}

Let $\NN=G/\Lambda$ be a Riemannian compact nilmanifold and $m$ its Haar probability measure.  An \emph{affine rational submanifold} $\mathscr{S}\subseteq \NN$ is a quotient $gH/(H\cap \Lambda)$ where $H$ is a connected closed subgroup of $G$ intersecting $\Lambda$ cocompactly, and $g\in G$ is a translation parameter. It carries a unique $gHg^{-1}$-invariant probability measure, denoted by $m_{\Su}$. 

The goal of the section is to reduce quantitative equidistribution of $(\Su, m_{\Su})$ in $(\NN, m)$ to a diophantine condition on the Lie algebra of $H$.

We call $\kh \subseteq \kg$ the respective Lie algebras of $H\subseteq G$ and write $\kt=\kg/[\kg, \kg]$ for the largest abelian quotient of $\kg$. The projection map from $\kg$ to $\kt$ is denoted by $\kg \rightarrow \kt, w\mapsto \overline{w}$ and  sends $\log \Lambda$ to a lattice $Z_{\Lambda}=\overline{\log \Lambda}$ in $\kt$. For future reference, we fix a basis of the latter, which leads to  identifications $\kt\equiv \R^d$, $Z_{\Lambda}\equiv \Z^d$.

\begin{prop.}[Equidistribution of  rational submanifolds] 
\label{equid}
Let $\NN=G/\Lambda$ be a Riemannian compact nilmanifold, and let $\theta\in (0,1]$. There exists $L>0$ such that for every $\delta\in (0,1/2)$, any  affine rational submanifold $\mathscr{S}\subseteq \NN$ satisfying
\[
\min \{\norm{q} \,:\, q\in \Z^d\setminus \{0\}, \,\langle q, \overline{\kh}\rangle =0\} \geq \left(\frac{1}{\delta}\right)^L
\]
is $\delta$-equidistributed with respect to $\HH^\theta(\NN)$, that is,  
\[
\Big|\int_{\Su} f dm_{\Su} -  \int_{\NN}f \,dm\Big|  \, \leq\,  \delta \norm{f}_{\HH^\theta}
\]
for all $f\in \HH^\theta(\NN)$.
\end{prop.}

We deduce this proposition from \Cref{th:GreenTao}, which is itself a direct corollary of
a much stronger result of Green and Tao, \cite[Theorem 1.16]{GreenTao12}.

\begin{thm}
\label{th:GreenTao}
Let $\NN=G/\Lambda$ be a Riemannian compact nilmanifold.
Then there is a constant $L>0$ such that, for every $\d\in(0,1/2)$, $g\in G$ and  $w\in \kg$, at least one of the following two statements is true.
\begin{enumerate}
\item The sequence $(g\exp(kw)\L)_{1\le k\le n}$ is $\d$-equidistributed
with respect to $\HH^1(\NN)$ for all sufficiently
large $n$.
That is to say, for all sufficiently large $n$, for all $f\in \HH^1(\NN)$, we have
\[
\Big|\frac{1}{n}\sum_{k=1}^{n}f(g\exp(kw)\Lambda) -  \int_{\NN}f dm\Big|   \leq  \delta \norm{f}_{\HH^1}.
\]

\item There is $q\in \Z^d\backslash \{0\}$ with $\|q\|< \d^{-L}$ and 
$\langle q,\overline w\rangle\in\Z$.
\end{enumerate}
\end{thm}

The above result does not use the full force of \cite[Theorem 1.16]{GreenTao12} in two
significant ways.
First, the result of Green and Tao is for general polynomial sequences in place of
$g\exp(kw)$.
Second, it is possible to control in a very efficient way how large $n$ needs to be in item 1
at the expense of replacing $\langle q,\overline w\rangle\in\Z$ in item 2 by 
an upper bound on $\dist(\langle q,\overline w\rangle,\Z)$.
These features are crucial in some other applications of \cite[Theorem 1.16]{GreenTao12},
but we do not need them.

\begin{proof}[Proof of Proposition \ref{equid}]
We  suppose without loss of generality that  $\theta=1$ (see \cite[proof of Theorem 2.3]{GorSpat15}). We choose $L>0$ as in  Theorem \ref{th:GreenTao} and consider an affine rational submanifold $\Su\subseteq \NN$ as well as a constant $\delta\in (0,1/2)$.

Assume that   $\Su$ is not $\delta$-equidistributed for $\HH^1(\NN)$. Fix a vector $w\in \kh$ such that $\norm{\overline{w}}\leq \frac{1}{2}\delta^L$ and whose projection in $\kh/[\kh, \kh]$ does not belong to any translate of a $H\cap \Lambda$-rational proper subspace
by a $H\cap \Lambda$-rational vector. By 
Theorem \ref{th:GreenTao}, or in this case even by an earlier theorem of Leon Green
\cite{Green61, AusGreenHahn64},
we have  the weak-$*$ convergence:  
 \[
 \frac{1}{n}\sum_{k=1}^{n}\delta_{g\exp(kw)\Lambda}\, \longrightarrow\, m_{\Su}.
 \]
In particular, for large enough $n$, the sequence $(g\exp(kw)\Lambda)_{0\leq k\leq n}$  does not satisfy $\delta$-equidistribution for $\HH^1(\NN)$ either. Theorem \ref{th:GreenTao} yields some integer vector $q\in \Z^d\setminus \{0\}$ such that $\norm{q}< \delta^{-L}$ and $\langle q,\overline{w}\rangle\in\Z$. This forces $\langle q,\overline{w}\rangle=0$ as  our choice of $w$ guarantees that $|\langle q,\overline{w}\rangle|\leq 1/2$. As the smallest rational subspace of $\kt$ containing $\overline{w}$ is $\overline{\kh}$, we get that $\langle q, \overline{\kh}\rangle=0$, which concludes the proof.
\end{proof}

\section{Diophantine estimates}

\label{Sec3}

Let $K$ be a number field.
We denote by $M_K$ the set of its places,  by $M_{K,\infty}$ its Archimedean places
and by $M_{K,f}$ its finite places.
We use the convention that we include complex places twice (one for each complex conjugate
embedding) and we write $|x|_v$ for $x\in K$ and $v\in M_{K,\infty}$ to denote
the usual absolute value arising from the embedding of $K$
in $\R$ or $\C$ associated to $v$.
(This convention is not standard, but it is also not unprecedented, see \cite[Chapter 14]{Mas}.)
The finite places will not play much role in this paper, so we do not specify
which convention is used for inclusion in $M_{K,f}$ or for the normalization
of the corresponding absolute values.
We do insist, however, that the product formula $\prod_{v\in M_K}|x|_v=1$
holds for all $x\in K$.

The projective height of a  tuple $(a_1,\ldots,a_n)\in K^n$ is denoted by
\[
H(a_1,\ldots,a_n)=\prod_{v\in M_K} \max_j(|a_j|_v^{1/[K:\Q]}).
\]
We note that this quantity is a projective notion, that is, it is invariant under
multiplication of each coordinate by the same scalar.
This fact is an immediate consequence of the product formula.

The ring of integers in $K$ is denoted by  $\cO(K)$ and  
for $h\in\R_{\ge 0}$, we set
\[
\cO(K)_{h}=\{a\in\cO(K):|a|_v\le h\text{ for all $v\in M_{K,\infty}$}\}.
\]

We write $\cO(K)^\times$ for the group of (multiplicative) units in $\cO(K)$.
For $n\in \Z_{\ge 2}$ and $u=(u_1,\ldots,u_n)\in (\cO(K)^\times)^n$, we write
\[
\a(u)=\min_{I\subset\{1,\ldots, n\}, |I|\ge 2}H(u_i:i\in I)^{1/(|I|-1)}.
\]

In what follows, we consider a generalization of the classical unit equation
$u_1+u_2=1$ to be solved for $u_1,u_2\in \cO(K)^\times$.
This subject has a rich literature, the recent book \cite{EG}
is a good general reference.

\begin{thm}
\label{unit-eq}
Let $K$ be a number field.
For all $\e>0$ and $n\geq 2$, there is an (ineffective) constant $r>0$
such that the following holds.
Let $u=(u_1,\ldots,u_n)\in (\cO(K)^\times)^n$ and let
$a_1,\ldots,a_n\in \cO(K)_{h}$ not all $0$ for some $h\in \R_{\ge 0}$.
Suppose
\[
a_1u_1+\ldots+a_n u_n=0.
\]
Then $h\ge r\a(u)^{1-\e}$.
\end{thm}

We comment on how this result is related to the classical unit equation.
Taking $n=3$, $a_1=a_2=1$, $a_3=-1$, $u_3=1$, Theorem \ref{unit-eq} implies
that the classical unit equation has finitely many solutions, a result that
goes back to Siegel, Mahler and Lang.

Theorem \ref{unit-eq} is not new.
It could be deduced (with a slightly weaker exponent) from
a result of Evertse \cite{Eve}, see also \cite[Theorem 6.1.1]{EG}.
We give a short proof based on \cite{Eve} for the reader's convenience.

As a simple application of Dirichlet's box principle, we also show below that the
exponent in the lower bound on $h$ is almost optimal in the sense that
$1-\e$ could not be replaced by a number larger than $1$.

\subsection{The subspace theorem}

At the heart of the proof of Theorem \ref{unit-eq} is W. Schmidt's subspace theorem which we now recall. 
\begin{thm}[{\cite[Corollary 7.2.5]{BomGub07}}]\label{th:subspace}
Let $K$ be a number field and let $\{|.|_{v}, v \in M_{K}\}$ be
a system of absolute values on $K$ as above. 
Let $n\geq 2$ and let $V$ be an $n$-dimensional vector space over $K$.
For each $v\in M_{K,\infty}$, let $\Lambda_1^{(v)},\ldots,\Lambda_n^{(v)}$ be a basis
of the dual space $V^*$.
Furthermore, let $\Lambda_1^{(0)},\ldots,\Lambda_n^{(0)}$ be another basis of $V^*$.
Then for all $\e>0$, the solutions of
\[
\prod_{v\in M_{K,\infty}}\prod_{j=1}^n|\Lambda_j^{(v)}(x)|_v
\le H(\Lambda_1^{(0)}(x),\ldots,\Lambda_n^{(0)}(x))^{-\e}
\]
for $x\in V$ with $\Lambda^{(0)}_j(x)\in\cO(K)$ for all $j=1,\ldots,n$
are contained in a finite union of proper subspaces of $V$.
\end{thm}

The functionals $\Lambda_1^{(0)},\ldots,\Lambda_n^{(0)}$ can be used to identify
$V$ with $K^n$ and using this identification, $\L_j^{(v)}$ can be identified with
a linear form for each $j$ and $v$.
Using these identifications, Theorem \ref{th:subspace}  is translated into the form in
\cite[Corollary 7.2.5]{BomGub07}.
In our application that follows,
there is no natural choice for an identification between $V$ and $K^n$,
so the above formulation will suit our purposes best.

\subsection{Proof of Theorem \ref{unit-eq}}

The proof is by induction on $n$.
Suppose first $n=2$.
Let $u_1,u_2\in\cO(K)^\times$ and $a_1,a_2\in\cO(K)_{h}$ for some $h\in\R_{\ge 0}$ such that  $a_{1}, a_{2}\neq 0$ and
$a_1u_1+a_2u_2=0$.
We observe that
\[
(u_1,u_2)=\big(\frac{u_1}{a_2}a_2,-\frac{u_1}{a_2}a_1\big),
\]
hence
\[
H(u_1,u_2)=H(a_1,a_2)\le (h^{1/[K:\Q]})^{[K:\Q]}=h.
\]
Since, $\a(u_1,u_2)=H(u_1,u_2)$, this proves our claim with $r=1$.

Now suppose $n\ge 3$, and that the claim holds for all smaller values of $n$. 
Let $V$ be the subspace of $K^n$ of those points that satisfy the
equation $x_1+\ldots+x_n=0$.
Observe that
\[
y:=(a_1u_1,\ldots,a_nu_n)\in V.
\]

For each place $v\in M_{K,\infty}$, let $\Lambda_1^{(v)},\ldots,\Lambda_{n-1}^{(v)}$
be an enumeration of all but one of the functionals $(x_1,\ldots,x_n)\mapsto x_j$
so that we omit one of the indices $j$ for which
$|u_j|_v$ is maximal.
We also set $\Lambda_j^{(0)}(x)=x_j$, say, for $j=1,\ldots,n-1$.
Observe that by the product formula
\[
\prod_{v\in M_{K,\infty}}\prod_{j=1}^{n-1} |\Lambda_j^{(v)}(u_1,\ldots, u_n)|_{v}
=\prod_{v\in M_{K,\infty}}\frac{\prod_{j=1}^{n}|u_j|_v }{\max_j(|u_j|_v)}
=H(u_1,\ldots,u_n)^{-[K:\Q]}.
\]

Using $a_1,\ldots,a_n\in\cO(K)_{h}$ and the definition of $\Lambda_j^{(v)}$, we have
\[
|\Lambda_j^{(v)}(a_1u_1,\ldots,a_nu_n)|_{v}\le h |\Lambda_j^{(v)}(u_1,\ldots,u_n)|_{v}
\]
for all $j$ and $v$.
Therefore,
\[
\prod_{v\in M_{K,\infty}}\prod_{j=1}^{n-1}|\Lambda_j^{(v)}(y)|_v
\le h^{(n-1)[K:\Q]} H(u)^{-[K:\Q]}.
\]

We assume $h\le \a(u)^{1-\e}$ as we may, for otherwise the claim is trivial.
Then $h\le H(u)^{(1-\e)/(n-1)}$, and 
we observe that
\begin{align*}
H(\Lambda_1^{(0)}(y),\ldots,\Lambda_{n-1}^{(0)}(y))
=&H(a_1u_1,\ldots,a_{n-1}u_{n-1})\\
\le& h H(u_1,\ldots,u_{n-1})
\le h H(u)
\le H(u)^2.
\end{align*}
Thus,
\[
\prod_{v\in M_{K,\infty}}\prod_{j=1}^{n-1}|\Lambda_j^{(v)}(y)|_v
\le H(u)^{-\e[K:\Q]}\le H(\Lambda_1^{(0)}(y),\ldots,\Lambda_{n-1}^{(0)}(y))^{-\e[K:\Q]/2}.
\]

Therefore, the subspace theorem applies and it follows that $(a_1u_1,\ldots, a_nu_n)$
is contained in finitely many proper subspaces of $V$.
In what follows, we use the induction hypothesis to show that for any proper subspace
$U$ of $V$, there is a constant $r=r(U)>0$ such that $h\ge r \a(u)^{1-\e}$ holds
whenever there are $a_1,\ldots,a_n$ and $u_1,\ldots, u_n$ that satisfy the assumptions
in the proposition with $y\in U$.
We can then take the minimum of $r(U)$ over all subspaces $U$ that arise through the
application of the subspace theorem taking into account all possible choices of the
functionals $\Lambda_j^{(v)}$, of which there are finitely many.
Then the claim will hold with this minimum in the role of $r$.

Let now $U$ be a proper subspace of $V$ with $y\in U$,
and let $b_1,\ldots,b_n\in \cO(K)$ be such that
\[
b_1x_1+\ldots+b_nx_n=0
\]
for all $(x_1,\ldots, x_n)\in U$ and not all $b_j$ are equal.
We also suppose that $a_j\neq0$ for all $j$, as we may, for otherwise we may omit the
$0$ coordinates, and the induction hypothesis applies directly.
We now observe that
\[
(b_1-b_n)a_1u_1+\ldots+(b_{n-1}-b_n)a_{n-1}u_{n-1}=0
\]
and that not all of the coefficients $(b_j-b_n)a_j$ vanish.

We also note that $(b_j-b_n)a_j\in \cO(K)_{Ch}$ for some constant $C$
depending on the $b_j$'s.
Therefore, we are in a position to apply the induction hypothesis
and conclude
\[
Ch\ge c\a(u_1,\dots,u_{n-1})^{1-\e}\ge r\a(u)^{1-\e},
\]
and this completes the proof.

\subsection{Optimality of the exponent}

We now prove the claimed optimality of Theorem \ref{unit-eq}.

\begin{prp}
Let notation be as in Theorem \ref{unit-eq}.
Then there is a constant $R>0$
such that for all 
$u=(u_1,\ldots,u_n)\in (\cO(K)^\times)^n$ there exist 
$a_1,\ldots,a_n\in \cO(K)_{h}$ not all $0$ such that
\[
a_1u_1+\ldots+a_n u_n=0
\]
and $h\le R\a(u)$.
\end{prp}
\begin{proof}
We assume, as we may, that $\a(u)=H(u)^{1/(n-1)}$, for otherwise
we can omit some coordinates from $u$ so that the identity holds.
	
By Dirichlet's unit theorem, there exists  $\lambda \in \cO(K)^\times$ such that
\[
\max(|\lambda u_1|_v,\ldots,|\lambda u_n|_v)\le C_{0}H(u)
\]
for all $v\in M_{K,\infty}$ and some $C_{0}>0$ depending only on $K$. Up  to replacing $u$ by  $\lambda u$, which does not affect the height, we may assume $\lambda=1$. It follows that for all $a_1,\ldots,a_n\in \cO(K)_{h}$, 
\[
a_1u_1+\ldots+a_n u_n\in \cO(K)_{C_{1}hH(u)}
\]
where $C_{1}=nC_{0}$.

We observe that
\[
c_2 h^{[K:\Q]}\le |\cO(K)_h|\le C_2 h^{[K:\Q]}
\]
for some constants $c_2,C_2>0$ depending only on $K$
provided $h$ is sufficiently large.
This means that there are at least $c_2^n (h/2)^{n[K:\Q]}$ choices
for $a_1,\ldots,a_n\in \cO(K)_{h/2}$ and
there are at most $C_2(C_{1}hH(u))^{[K:\Q]}$ possible values for
$a_1 u_1+\ldots+a_n u_n$.

Now we take $h=R \,H(u)^{1/(n-1)}$ for a suitably large constant $R$ so that
\[
c_{2}^n (h/2)^{n[K:\Q]}>C_2(C_{1}hH(u))^{[K:\Q]}.
\]
Dirichlet's box principle implies that there are
$b, \wt b \in (\cO(K)_{h/2})^n$ 
such that $b\neq \wt b$ and
\[
b_1u_1+\ldots+b_n u_n=\wt b_1 u_1+\ldots+\wt b_n u_n.
\]
 The claim follows by taking $a_j=b_j-\wt b_j$.
\end{proof}

\section{Growth of eigencharacters}

\label{Sec1}

Let $\alpha:\Z^l \rightarrow \Aut(\NN)$ be a morphism from $\Z^l$ to the group of automorphisms of a compact nilmanifold $\NN=G/\Lambda$. We assume that $\alpha(z)$ acts ergodically on $\NN$ for every $z\neq0$ and estimate the growth of the  eigencharacters of $\alpha$ acting on the abelianized Lie algebra of $G$. 

 As in Section \ref{Sec2}, we set $\kt=\kg/[\kg,\kg]$ and identify $\kt$ with some $\R^d$ so that the projection $Z_{\Lambda}$ of $\log \Lambda$ in $\kt$ corresponds to $\Z^d$. The representation $\alpha$ induces  by differentiation a representation of $\Z^l$ on $\kt$ which preserves $Z_{\Lambda}$. We call it $d_\kt\alpha : \Z^l\rightarrow \GL^{\pm1}_{d}(\Z)$, where $\GL^{\pm1}_{d}(\Z)$ denotes the set of  $d\times d$-matrices with coefficients in $\Z$ and determinant $\pm1$. As $\Z^l$ is  abelian, such a representation is triangularizable over $\QQ$ : we can decompose $\QQ^d$ as a direct sum of subrepresentations 
\[
\QQ^d=\oplus_{\chi \in \mathscr{X}}L_{\chi}
\] 
where $\mathscr{X}$ is a (finite) set of characters $\chi  : \Z^l \rightarrow \QQ^\times$, and for each $\chi \in \mathscr{X}$,  
the generalized eigenspace 
\[
L_{\chi}:=\{v\in \QQ^d \,|\, \forall z\in \Z^l, \,\left(d_{\kt}\alpha(z)-\chi(z) \text{Id}\right)^d(v)=0 \} 
\]
is nonzero.  
Moreover, the Galois group $\G=\text{Gal}(\QQ/\Q)$ acting coordinate-wise on  $\QQ^d$ permutes these eigenspaces according to the formula $\sigma . L_{\chi}=L_{\sigma \circ \chi}$. 

\bigskip
The assumption of ergodicity on $\alpha$ can be reformulated as follows. 
\begin{lemma}[Exponential growth of eigencharacters] \label{exp}
	
	There exists $c>0$ such that for every $\chi_{0} \in \mathscr{X}$, and for every $z\in \Z^l\smallsetminus \{0\}$, 
	\[
	\max_{\chi \in \G.\chi_{0}}|\chi (z)|\geq e^{c\norm{z}}.
	\]
\end{lemma}  

\begin{proof} We recall the proof given in  \cite[Lemma 2.1]{GorSpat15}. Let $\chi_{0}\in \mathscr{X}$. We claim that the morphism of $\Z$-modules  
	\[
	\psi : \Z^l\rightarrow \R^{|\G.\chi_{0}|}, \,z\mapsto (\log |\chi(z)|)_{\chi\in \G.\chi_{0}} 
	\]
	is injective with discrete image. Its  $\R$-linear extension $\overline{\psi} : \R^l\rightarrow \R^{|\G.\chi_{0}|}$  is thus  injective, whence satisifies 
	$\norm{\overline{\psi}(z)} \geq c_{\chi_{0}}\norm{z} $ for some $c_{\chi_{0}}>0$ and every $z\in \Z^l$. The lemma follows by finiteness of $\mathscr{X}$. 
	
	To prove the claim, we argue by contradiction assuming the existence of a sequence $(z_{i})\in (\Z^l)^\N$ of non-zero vectors such that $\psi(z_{i})\rightarrow 0$. This means that $|\chi(z_{i})|\rightarrow 1$ for every $\chi\in \G.\chi_{0}$. In particular, the associated minimal polynomials $P_{i}=\prod_{\chi\in \G.\chi_{0}}(X-\chi(z_{i}))\in \Z[X]$ have bounded coefficients, hence belong to a finite subset of $\Z[X]$. Necessarily, the sequence of vectors $(\chi(z_{i}))_{\chi\in \G.\chi_{0}}$ has a constant subsequence,  yielding some $z\in \Z^l\smallsetminus 0$ such that $|\chi(z)|=1$ for every $\chi\in \G.\chi_{0}$. By Kronecker's theorem (see e.g. \cite[Theorem 1.5.9]{BomGub07}), $\chi_{0}(z)$ is a root of unity. This contradicts the ergodicity of the toral automorphism  induced by $\alpha(z)$ on $G/[G,G]\Lambda$, whence that of $\alpha(z)$ on $G/\Lambda$.

\end{proof}

\section{Proof of exponential $n$-mixing}
\label{Sec4}

We now proceed to the proof of Theorem \ref{TH}. 
Fix $n \geq 2$, $\theta \in (0, 1]$. Let  $\z=(z_{1}, \dots, z_{n}) \in (\Z^l)^n$ and consider some $\theta$-H\"older functions  $f_{1}, \dots, f_{n}\in \HH^{\theta}(\NN)$ as well as translation parameters $g_{1}, \dots, g_{n}\in G$. We begin with two observations:

\begin{itemize}
\item The integral at study can be rewritten as
\begin{align*}
\int_{\NN}\prod_{i=1}^nf_{i}(g_{i}\alpha(z_{i})x)\,dm(x)=\int_{\Su}f_{1}\otimes \dots\otimes f_{n} \,dm_{\Su}
\end{align*}
where $\Su$ is an affine rational submanifold of $\NN^n$, defined as the image of the embedding 
\[
\NN\rightarrow \NN^n, \,x\mapsto (g_{1}\alpha(z_{1})x, \dots, g_{n}\alpha(z_{n})x).
\]

\item Equiping $\NN^n$ with the product Riemannian metric on $\NN$, we have  
\[
\norm{f_{1}\otimes \dots\otimes f_{n}}_{\HH^{\theta}}\leq \norm{f_{1}}_{\HH^{\theta}}\dots \norm{f_{n}}_{\HH^{\theta}}.
\]

\end{itemize}

These observations reduce \Cref{TH} to showing that the submanifold $\Su$ is $\delta$-equidistributed with respect to $\HH^{\theta}(\NN^n)$ for some $\delta<\frac{C}{N(\z)^\eta}$ where $N(\z)=\exp(\min_{1\leq i\neq j\leq n}\norm{z_{i}-z_{j}})$ and $C, \eta>0$ are constants allowed to depend on the initial data $(G,\Lambda, \alpha, n, \theta)$ or possibly related structures, but not on $\z$ nor the $g_{i}$'s.

Denote by  $L>0$ the constant associated to the product nilmanifold $\NN^n$ and $\theta$ as in \Cref{equid}. Let $\delta\in (0, 1/2)$ be such that $\Su$ is not $\delta$-equidistributed in $\NN^n$. By   \Cref{equid}, there exists $q=(q_{1}, \dots, q_{n})\in (\Z^d)^n$ such that 
\[
0< \norm{q}<  \left(\frac{1}{\delta}\right)^L  \,\,\,\text{ and }\,\,\, \langle q, d_{\kt}\alpha(\z)\kt\rangle=0
\]
where for $\omega\in \kt$, we set $d_{\kt}\alpha(\z)\omega=(d_{\kt}\alpha(z_{1})\omega, \dots, d_{\kt}\alpha(z_{n})\omega)\in \kt^n$,
and recall that $d_{\kt}\a(z)$ is the projection of the differential of $\a(z)$ to $\kt$.
Taking $\QQ$-linear combinations of the above equality, we get for every $\omega\in \QQ^d$,
\begin{align}\label{null}
\langle q, d_{\kt}\alpha(\z)\omega\rangle=0.
\end{align}
Let us choose once and for all a basis $\beta_{\chi}$ of each generalized eigenspace $L_{\chi}$ for $\chi \in \mathscr{X}$ whose vectors have algebraic integer coefficients and such that $d_{\kt}\alpha(\Z^l)$ is represented by upper-triangular matrices.  As the $L_{\chi}$'s span $\QQ^d$ and $q$ is non-zero, there must exist  $\chi_{0}\in \mathscr{X}$ such that $ \langle q_{i}, L_{\chi_{0}}\rangle \neq \{0\}$ for some $i\in \{1, \dots ,n\}$. We let $\omega_{0}$ be the first element of the basis $\beta_{\chi_{0}}$ such that $ \langle q_{i}, \omega_{0}\rangle \neq 0$ for some $i$. We can then write for every $i$ 
\[
 \langle q_{i}, d_{\kt}\alpha(z_{i})\omega_{0}\rangle =\chi_{0}(z_{i}) \langle q_{i} , \omega_{0}\rangle
\]
and equation \eqref{null} with $\omega=\omega_{0}$ yields 
\[
\sum_{i=1}^n\chi_{0}(z_{i}) \langle q_{i}, \omega_{0}\rangle=0.
\]

We let $K$ be the number field  generated by the coefficients of the vectors belonging to the  basis $(\beta_{\chi})_{\chi \in \mathscr{X}}$.
For every $i$, we then have $\langle q_{i}, \omega_{0}\rangle\in \cO(K)$  and $\chi_{0}(z_{i})\in \cO(K)^\times$ (eigenvalues of the matrix $d_{\kt}\alpha(z_{i})\in\GL_{d}^{\pm1}(\Z)$ are algebraic units, and the relation $d_{\kt}\alpha(z_{i})\omega_{0}=\chi_{0}(z_{i})w_{0}$ forces $\chi_{0}(z_{i})$ to be in $K$). Fix $\eps\in (0,1)$ arbitrarily. \Cref{unit-eq} yields a constant $r>0$ depending only on $(K, n, \eps)$ such that, for some $i_{0}\in \{1, \dots, n\}$, and some Galois automorphism $\sigma\in \G$,
\begin{align} \label{ineq}
|\langle q_{i_{0}}, \sigma(\omega_{0})\rangle|\geq r \alpha(u)^{1-\eps}
\end{align}
where $u=(\chi_{0}(z_{1}), \dots, \chi_{0}(z_{n}))$. The exponential growth of eigencharacters presented in \Cref{exp} yields a lower bound on $\alpha(u)$:
\begin{align*}
\alpha(u)\,&\geq \min_{1\leq i\neq j\leq n}\left(H(\chi_{0}(z_{i}), \chi_{0}(z_{j}))\right)^{1/(n-1)}\\
&=\min_{1\leq i\neq j\leq n} \prod_{v\in M_{K, \infty}} \max (|\chi_{0}(z_{i}-z_{j})|_{v}, 1)^\frac{1}{(n-1)[K:\Q]}\\
&\geq \min_{1\leq i\neq j\leq n} \exp{\left(\frac{c \norm{z_{i}-z_{j}}}{(n-1)[K:\Q]}\right)}.
\end{align*}

Plugging this lower bound into the inequality (\ref{ineq}), and using the Cauchy-Schwarz inequality, we get 
\[
\norm{q}\geq r c_{1}   \exp{\left(\frac{c(1-\eps)}{(n-1)[K:\Q]}  \min_{1\leq i\neq j\leq n}\norm{z_{i}-z_{j}}\right)}
\] 
where $c_{1}=\min_{\sigma\in \G}\,\norm{\sigma(\omega_{0})}^{-1}$ only depends on our choice of $(\beta_{\chi})_{\chi\in \mathscr{X}}$. Recalling that $\delta^{-L}> \norm{q}$, we obtain that 
\[
\delta<   r c_{1} N(\z)^{-\eta} \,\,\,\,\,\text{ where }\,\,\,\,\, \eta:=\frac{c(1-\eps)}{(n-1)[K:\Q]L}.
\]

To sum up the above discussion, we have  proven that for any $\z\in (\Z^l)^n$ satisfying $r c_{1} N(\z)^{-\eta} <1/2$, the submanifold $\Su$ of $\NN^n$ is $r c_{1} N(\z)^{-\eta}$-equidistributed for $\HH^{\theta}(\NN^n)$. It follows that for all $\z\in (\Z^l)^n$, all $f_{1}, \dots, f_{n}\in \HH^{\theta}(\NN)$, $g_{1}, \dots, g_{n}\in G$, we have 
\[
\Big|\int_{\NN}\prod_{i=1}^n f_{i}(g_{i}\alpha(z_{i})x) \,dm(x)\,\, -\,\, \prod_{i=1}^n \int_{\NN}f_{i}(x) \,dm(x)\Big|   \,\,\leq \,\, \frac{4rc_{1}}{ N(\z)^{\eta}} \prod_{i=1}^n \norm{f_{i}}_{\HH^\theta}
\]
and this concludes the proof. 
(The constant $4$ on the right is needed to get a
valid statement when $r c_{1} N(\z)^{-\eta} \ge 1/2$.)

\bibliographystyle{abbrv}

\bibliography{bibliography}

\begin{thebibliography}{10}

\bibitem{AusGreenHahn64}
L.~Auslander, L.~Green, and F.~Hahn.
\newblock {\em Flows on homogeneous spaces}.
\newblock Annals of Mathematics Studies, No. 53. Princeton University Press,
  Princeton, N.J., 1963.
\newblock With the assistance of L. Markus and W. Massey, and an appendix by L.
  Greenberg.

\bibitem{BomGub07}
E.~Bombieri and W.~Gubler.
\newblock {\em Heights in {D}iophantine geometry}, volume~4 of {\em New
  Mathematical Monographs}.
\newblock Cambridge University Press, Cambridge, 2006.

\bibitem{Dolg04}
D.~Dolgopyat.
\newblock Limit theorems for partially hyperbolic systems.
\newblock {\em Trans. Amer. Math. Soc.}, 356(4):1637--1689, 2004.

\bibitem{Eve}
J.-H. Evertse.
\newblock On sums of {$S$}-units and linear recurrences.
\newblock {\em Compositio Math.}, 53(2):225--244, 1984.

\bibitem{EG}
J.-H. Evertse and K.~Gy\H{o}ry.
\newblock {\em Unit equations in {D}iophantine number theory}, volume 146 of
  {\em Cambridge Studies in Advanced Mathematics}.
\newblock Cambridge University Press, Cambridge, 2015.

\bibitem{FJ}
F.~T. Farrell and L.~E. Jones.
\newblock Anosov diffeomorphisms constructed from {$\pi \sb{1}\,{\rm
  Diff}\,(S\sp{n})$}.
\newblock {\em Topology}, 17(3):273--282, 1978.

\bibitem{FKS1}
D.~Fisher, B.~Kalinin, and R.~Spatzier.
\newblock Totally nonsymplectic {A}nosov actions on tori and nilmanifolds.
\newblock {\em Geom. Topol.}, 15(1):191--216, 2011.

\bibitem{FKS2}
D.~Fisher, B.~Kalinin, and R.~Spatzier.
\newblock Global rigidity of higher rank {A}nosov actions on tori and
  nilmanifolds.
\newblock {\em J. Amer. Math. Soc.}, 26(1):167--198, 2013.
\newblock With an appendix by James F. Davis.

\bibitem{GorSpat14}
A.~Gorodnik and R.~Spatzier.
\newblock Exponential mixing of nilmanifold automorphisms.
\newblock {\em J. Anal. Math.}, 123:355--396, 2014.

\bibitem{GorSpat15}
A.~Gorodnik and R.~Spatzier.
\newblock Mixing properties of commuting nilmanifold automorphisms.
\newblock {\em Acta Math.}, 215(1):127--159, 2015.

\bibitem{GreenTao12}
B.~Green and T.~Tao.
\newblock The quantitative behaviour of polynomial orbits on nilmanifolds.
\newblock {\em Ann. of Math. (2)}, 175(2):465--540, 2012.

\bibitem{Green61}
L.~W. Green.
\newblock Spectra of nilflows.
\newblock {\em Bull. Amer. Math. Soc.}, 67:414--415, 1961.

\bibitem{Lind82}
D.~A. Lind.
\newblock Dynamical properties of quasihyperbolic toral automorphisms.
\newblock {\em Ergodic Theory Dynam. Systems}, 2(1):49--68, 1982.

\bibitem{Mas}
D.~Masser.
\newblock {\em Auxiliary polynomials in number theory}, volume 207 of {\em
  Cambridge Tracts in Mathematics}.
\newblock Cambridge University Press, Cambridge, 2016.

\bibitem{MilWard11}
R.~Miles and T.~Ward.
\newblock A directional uniformity of periodic point distribution and mixing.
\newblock {\em Discrete Contin. Dyn. Syst.}, 30(4):1181--1189, 2011.

\bibitem{Parry69}
W.~Parry.
\newblock Ergodic properties of affine transformations and flows on
  nilmanifolds.
\newblock {\em Amer. J. Math.}, 91:757--771, 1969.

\bibitem{RHW}
F.~Rodriguez~Hertz and Z.~Wang.
\newblock Global rigidity of higher rank abelian {A}nosov algebraic actions.
\newblock {\em Invent. Math.}, 198(1):165--209, 2014.

\bibitem{SchmWard93}
K.~Schmidt and T.~Ward.
\newblock Mixing automorphisms of compact groups and a theorem of
  {S}chlickewei.
\newblock {\em Invent. Math.}, 111(1):69--76, 1993.

\end{thebibliography}

\bigskip
\noindent\textsc{Centre for Mathematical Sciences, Wilberforce Road, Cambridge CB3 0WB, UK}

\bigskip

\noindent\textit{Email address}, T. B\'enard: \texttt{tb723@cam.ac.uk}, P. P. Varj\'u: \texttt{pv270@dpmms.cam.ac.uk}

\end{document}